\documentclass[12pt]{amsart}
\usepackage{epsfig, amsmath, amssymb, latexsym,  amscd}
\usepackage{color}

\usepackage{tikz}
\usetikzlibrary{snakes}
\usepackage[colorlinks=true, linkcolor=blue, anchorcolor=blue, citecolor=blue, filecolor=blue, menucolor= blue, urlcolor=blue]{hyperref}
\usepackage{amsmath,amssymb,amsthm,amscd}
\usepackage{verbatim}
\usepackage{comment}
\usepackage{multirow}
\usepackage{mathtools}
\usepackage{enumitem}

\def\cocoa{{\hbox{\rm C\kern-.13em o\kern-.07em C\kern-.13em o\kern-.15em A}}}

\newtheorem{theorem}{Theorem}[section]
 \newtheorem{corollary}[theorem]{Corollary}
 \newtheorem{lemma}[theorem]{Lemma}
 \newtheorem{conjecture}[theorem]{Conjecture}
 \newtheorem{proposition}[theorem]{Proposition}
 \newtheorem{question}[theorem]{Question}
 
 \theoremstyle{definition}
 
 \theoremstyle{remark}
 \newtheorem{remark}[theorem]{Remark}
 
  \newtheorem{example}[theorem]{Example}

  \definecolor{MyDarkGreen}{cmyk}{0.7,0,1,0}

\author[J.~Migliore]{Juan Migliore}
\address{Department of Mathematics, University of Notre Dame, Notre Dame, IN 46556}
\email{migliore.1@nd.edu}

\author[F.\ Zanello]{Fabrizio Zanello}
\address{Department of Mathematical Sciences, Michigan Tech, Houghton, MI 49931}
\email{zanello@mtu.edu}

\title[Unimodal Gorenstein $h$-vectors without the Stanley-Iarrobino property]{Unimodal Gorenstein $h$-vectors without the Stanley-Iarrobino property}

\begin{document}

\begin{abstract} 
The study of the $h$-vectors of graded Gorenstein algebras is an important topic in combinatorial commutative algebra, which despite the large amount of literature produced during the last several years, still presents many interesting open questions. In this note, we commence a study of those unimodal Gorenstein $h$-vectors that do \emph{not} satisfy the Stanley-Iarrobino property. Our main results, which are characteristic free, show that such $h$-vectors exist: 1) In socle degree $e$ if and only if $e\ge 6$; and 2) In every codimension five or greater. The main case that remains open is that of codimension four, where no Gorenstein $h$-vector is known without the Stanley-Iarrobino property.\\
We conclude by proposing the following very general conjecture: The existence of any arbitrary level $h$-vector is \emph{independent} of the characteristic of the base field.
\end{abstract}

\keywords{Gorenstein $h$-vector, SI-sequence; unimodality; Hilbert function; level $h$-vector; Macaulay's theorem; $O$-sequence; artinian algebra}
\subjclass[2010]{Primary: 13D40; Secondary: 13H10, 13E10, 05E40}

\maketitle


\section{Introduction}

The goal of this note is to continue an investigation into the structure of the $h$-vectors of Gorenstein artinian algebras. A large amount of work has been done over the years to study the properties of Gorenstein rings, both for the substantial interest of the topic within commutative algebra, and for the many connections that it carries with other mathematical disciplines, such as combinatorics and algebraic geometry (see, as a highly nonexhaustive list, \cite{A,bass,Hu,P,St-80,St0,St2}). However, much is still unknown today on which integer sequences $h=(1,h_1,\dots, h_{e-1},h_e=1)$  may arise as Gorenstein $h$-vectors, if the codimension $r=h_1$ is greater than 3.

Recall that a sequence $h=(1,h_1,\dots,h_e)$ is \emph{symmetric} if $h_i=h_{e-i}$, for all indices $i$. We say that $h$ is \emph{differentiable} if its \emph{first difference}, $(1,h_1-1,\dots,h_e-h_{e-1})$, is the $h$-vector of some graded artinian algebra (see the next section for the relevant definitions and for Macaulay's characterization of such $h$-vectors). Finally, $h$ satisfies the \emph{Stanley-Iarrobino property}, or more briefly, is an \emph{SI-sequence}, if $h$ is symmetric and its \emph{first half}, $(1,h_1,\dots,h_{\lfloor e/2\rfloor})$, is differentiable. Note that an SI-sequence is always \emph{unimodal}, i.e., it never increases strictly after a strict decrease. For the importance of the concept of unimodality in algebra, combinatorics, and related areas, see the two classical survey articles \cite{Bre,Stan}.

It is well known that all Gorenstein  $h$-vectors are symmetric. It is also known that any SI-sequence is a Gorenstein $h$-vector \cite{CI,ha,MN}, while the converse has been shown to hold in codimension $r\le 3$ (if $r\le2$ the result is trivial; for $r=3$, see Stanley \cite{stanley} and then the second author \cite{Za3}). However, as proven by Stanley and subsequently other authors, there even exist Gorenstein $h$-vectors that are not {unimodal}, in any codimension $r\ge 5$ (see, just as a sample, \cite{reniar, Bo, BL, MNZ1, MNZ2, MZ}). This fact alone leads most experts to believe that a complete characterization of the possible Gorenstein $h$-vectors is probably hopeless.

In this note, we investigate Gorenstein  $h$-vectors  under a different perspective, by beginning a study of those that are unimodal but do \emph{not} satisfy the Stanley-Iarrobino property. This problem appears to have not yet been addressed in the literature, with the exception of a brief section in the second author's thesis \cite{Za0} (not published elsewhere). Our main results, which are all characteristic free, show that there exist unimodal, non-SI Gorenstein $h$-vectors of socle degree $e$ if and only if $e\ge 6$; and that there exist unimodal, non-SI Gorenstein $h$-vectors in every codimension $r\ge 5$. The main case that remains open is that of codimension $r=4$, where all Gorenstein $h$-vectors currently known are SI-sequences.

We wrap up our paper by conjecturing an intriguing and general fact on arbitrary level $h$-vectors: Namely, their existence is \emph{independent} of the characteristic of the base field.

\section{Notation and preliminary facts}

We  now recall a few standard definitions and preliminary facts. Let $R = k[x_1,\dots,x_r]$, where $k$ is any infinite field, and let  $\mathfrak{m} = (x_1, \dots ,x_r)$ be its maximal homogeneous ideal. We consider  standard graded artinian $k$-algebras $A=R/I=\bigoplus _{i\ge 0}[A]_{i}$ where $I$ is a homogeneous ideal of $R$. Since we may assume, without loss of generality, that $I$ does not contain nonzero forms of degree 1,  $r$ is the \emph{codimension} of $A$.

The \emph{Hilbert function} of $A$ is defined by $h_A(i)=\dim _k [A]_{i}$, for all $i\ge 0$. Notice that the Hilbert function of an artinian $k$-algebra has finite support and is captured in its {\em $h$-vector} $h = (1,h_1,\dots,h_e)$, where $e$ is the last index such that $h_{e}=h_A(e)>0$. The integer $e$ is referred to as the {\em socle degree} of $A$.

The \emph{socle}  of $A$ is the annihilator of $\overline{\mathfrak{m}} = (\overline{x_1}, \dots ,\overline{x_r})\subset A$; that is, $soc(A)=\{ a\in A \mid a\overline{\mathfrak{m}}=0 \}$. We say that $A$ is {\em level of type $t$} if the socle of $A$ is a $t$-dimensional $k$-vector space and is all concentrated in degree $e$. Finally, $A$ is {\em Gorenstein} if $A$ is level of type 1.

The following classical result of Macaulay characterizes the  integer sequences that arise as $h$-vectors of standard graded artinian algebras. For $n$ and $i$  positive integers, define the \emph{$i$-binomial expansion of $n$} as
$$n_{(i)} = \binom{n_i}{ i}+\binom{n_{i-1}}{ i-1}+...+\binom{n_j}{ j},$$
where $n_i>n_{i-1}>...>n_j\geq j\geq 1$. Such an expansion always exists and is unique (see, e.g., \cite{BH}, Lemma 4.2.6). Then let
$$ n^{<i>}=\binom{n_i+1}{ i+1}+\binom{n_{i-1}+1}{ i-1+1}+...+\binom{n_j+1}{ j+1}.$$

\begin{theorem}[Macaulay]\label{macaulay} A sequence $h=(1,h_1,\dots,h_e)$ of positive integers is the $h$-vector of some standard graded artinian algebra if and only if, for all indices $i$, 
$$h_{i+1}\leq h_i^{<i>}.$$
\end{theorem}

The sequences that satisfy Theorem \ref{macaulay} are usually referred to as \emph{$O$-sequences}. 

We now briefly recall a useful tool in the study of artinian $h$-vectors, namely \emph{Macaulay's inverse systems}. Given a  homogeneous ideal $I\subset R=k[x_1,\dots,x_r]$, define its \emph{inverse system} to be the graded $R$-submodule $M\subset
S=k[y_1,\dots,y_r]$, where $R$ acts on $S$ by contraction and $I=\operatorname{ann}(M)$. The external product that defines the module $S$ consists of the linear action uniquely determined by
$$ x_i \circ y_1^{a_1}y_2^{a_2}\cdots y_r^{a_r} =
  \begin{cases}
    y_1^{a_1}y_2^{a_2}\cdots y_i^{a_i-1}\cdots y_r^{a_r},&\text{if $a_i>0$,}\\
    0,&\text{if $a_i=0$,}
  \end{cases}
$$
for any nonnegative integers $a_1,a_2,\dots,a_r$. We may sometimes simplify the notation by considering inverse system submodules to also lie inside the  polynomial ring $R$. 

We refer the reader to \cite{ger,IKa} for more information and a comprehensive treatment of the theory of inverse systems. In particular, we note that when the characteristic of $k$ is zero, then one can use \emph{differentiation} in place of the above linear action, so that the $x_i$ act as partial derivatives. In general, if $R/I$ is artinian and  level of type $t$ and socle degree $e$, then its corresponding module $M \subset S$ is generated by $t$ linearly independent homogeneous forms of degree $e$. Also important for us below,  the Hilbert functions of $R/I$ and  $M$ coincide; that is, $\dim_k [R/I]_i = \dim_k [M]_i$, for all $i$.

We end this section with some constructions that will be crucial to our main results.

The first of these constructions, called \emph{trivial extensions},  is due to Stanley (\cite{stanley}, Example 4.3; see also \cite{Re}). Nearly four decades later, this simple idea remains the single most powerful tool that people have for constructing new Gorenstein $h$-vectors, in particular nonunimodal ones.

\begin{lemma}[Stanley]\label{te}
Given a level algebra with $h$-vector $(1,h_1,\dots,h_{e-1})$, there exists a Gorenstein algebra with $h$-vector $(1,H_1,\dots,H_{e-1},H_e=1)$, where for each $i=1,\dots,e-1$,
$$H_i=h_i+h_{e-i}.$$
\end{lemma}

The following is a  useful construction due to Iarrobino (\cite{iarrobino}, essentially Proposition 3.3; see also  \cite{fl}). Iarrobino's result can be phrased as follows.

\begin{proposition}[Iarrobino]\label{compressed}
Let $M\subset S=k[y_1\dots,y_r]$ be the inverse system module of a level algebra with $h$-vector $(1,h'_1,\dots, h'_e)$, and let $F\in S$ be a degree $e$  \emph{general} form (with respect to the Zariski topology). Then the $h$-vector of the level algebra corresponding to the inverse system module $\langle M,F\rangle$ is $(1,h_1,\dots, h_e)$, where for all $i$,
$$h_i=\min \left\{\binom{r-1+i}{i}, h'_i+\binom{r-1+e-i}{e-i}\right\}.$$
\end{proposition}

Finally, we recall a well-known method to, given a Gorenstein sequence of codimension $r$, produce a new Gorenstein sequence of codimension $r+1$ with similar properties. 

\begin{lemma} \label{increase codim}

If a Gorenstein algebra with $h$-vector $(1,H_1,H_2,\dots,H_{e-1},1)$ corresponds to the cyclic inverse system module $\langle F\rangle \subset k[y_1,\dots,y_r]$, then the cyclic module $\langle F+y_{r+1}^e\rangle \subset k[y_1,\dots,y_r,y_{r+1}]$ yields a Gorenstein algebra with $h$-vector
$$(1,H_1,H_2,\dots,H_{e-1},1)+(0,1,1,\dots,1,0)=(1,H_1+1,H_2+1,\dots,H_{e-1}+1,1).$$

\end{lemma}


\section{Main results}

In this section we present the main results of this paper, which contribute to our understanding of the possible codimensions and socle degrees of unimodal, non-SI Gorenstein $h$-vectors.

\begin{theorem}\label{e}
There exists a unimodal, non-SI Gorenstein $h$-vector of socle degree $e$ if and only if $e\ge 6$.
\end{theorem}

\begin{proof}
By symmetry, all Gorenstein $h$-vectors of socle degree $e=3$ are of the form $h=(1,r,r,1)$, and therefore are clearly SI-sequences. When $e=4$, a Gorenstein sequence $h$ is of the form $h=(1,r,a,r,1)$, where $r\le a\le \binom{r+1}{2}$ if $h$ is unimodal. Hence, the first difference of the first half of $h$ is $(1,r-1,a-r)$, and it is immediate to verify that it satisfies Macaulay's Theorem \ref{macaulay}. The proof for $e=5$, where $h$ Gorenstein must be of the form $(1,r,b,b,r,1)$, is entirely similar and will be omitted.

Now let $e\ge 6$. We begin with the $h$-vector
$$(1,h'_1=2,h'_2=3,\dots,h'_{e-1}=e),$$
which is level since it is given by the algebra $A$ obtained as a truncation of the polynomial ring $k[x,y]$ after degree $e-1$.

Consider a degree $e-1$ general form $F$ in $k[x,y,z]$. By Proposition \ref{compressed}, the inverse system module $\langle M,F\rangle $ yields a level $h$-vector $(1,h_1,\dots,h_{e-1})$, where for all $i$,
$$h_i=\min \left\{\binom{i+2}{2},(i+1)+\binom{e-i+1}{2}\right\}.$$

Therefore, by Lemma \ref{te}, we obtain a Gorenstein $h$-vector $(1,H_1,\dots,H_{e-1},H_e=1)$, with $H_i=h_i+h_{e-i}$ for all indices $i$. It is easy to see that $H$ is unimodal, since by construction, $h_i\ge h_{i-1}$ and $h_{e-i}\ge h_{e-(i-1)}$ for all  $i\le e/2$.

Also notice that, since $e\ge 6$,
$$H_1=3+\left(e+\binom{2}{2}\right)=e+4;$$
$$H_2=6+\left(e-1+\binom{3}{2}\right)=e+8;$$
$$H_3=10+\left(e-2+\binom{4}{2}\right)=e+14.$$

It follows that the first difference of $H$ begins with
$$(1,(e+4)-1,(e+8)-(e+4),(e+14)-(e+8),\dots)=(1,e+3,4,6,\dots),$$
in violation of Macaulay's Theorem \ref{macaulay} from degree 2 to 3. Thus, $H$ is not an SI-sequence.
\end{proof}

From the proof of  Theorem \ref{e}, we can in fact deduce the following.

\begin{corollary}\label{e4}
For any integers $e\ge 6$ and $r\ge e+4$, there exists a unimodal, non-SI Gorenstein $h$-vector of codimension $r$ and socle degree $e$.
\end{corollary}

\begin{proof}
The Gorenstein $h$-vectors $H$ constructed in the proof of Theorem \ref{e} have codimension $r=H_1=e+4$, for any $e\ge 6$. 
Then, using Lemma \ref{increase codim}, one can easily produce a unimodal, non-SI Gorenstein $h$-vector of any codimension  $r\ge e+4$ and socle degree $e$, as desired.
\end{proof}

\begin{example}\label{10}
The simplest example of a unimodal, non-SI Gorenstein $h$-vector given by the construction of Theorem \ref{e} is the following, of socle degree $e=6$:
$$(1,10,14,20,14,10,1).$$
By Lemma \ref{increase codim},
$$(1,a+10,a+14,a+20,a+14,a+10,1)$$
is therefore also a unimodal, non-SI Gorenstein $h$-vector, for any integer $a\ge 0$.
\end{example}

\begin{theorem}\label{r}
There exists a unimodal, non-SI Gorenstein $h$-vector in any codimension $r\ge 5$.
\end{theorem}

\begin{proof}
By Lemma \ref{increase codim}, it suffices to show the result for $r=5$. In fact, we will produce an infinite family of unimodal, non-SI Gorenstein $h$-vectors of codimension five and socle degree $e$, for any $e\ge 20$.

We begin with $e$ odd. Fix an integer $d\ge 10$, and consider the $h$-vector $h = (1,h_1,\dots,h_{2d})$ where
\[
\begin{array}{lcl}
h_1 & = & 3, \\
h_2& = & 6, \\
& \vdots \\
h_d & = & \binom{d+2}{2}, \\
h_{d+1}& = & \binom{d+2}{2}+1, \\
h_{d+2}& = & \binom{d+2}{2}+2, \\
h_{d+3}& = & \binom{d+2}{2}+3, \\
h_{d+4} & = & \min \left\{\binom{d+2}{2}+3,2\binom{d-2}{2}\right\}, \\
& \vdots \\
h_{2d-j} & = & \min \left\{\binom{d+2}{2}+3,2\binom{j+2}{2}\right\}, \\
& \vdots \\
h_{2d-1} & = & 6, \\
h_{2d} & = & 2.
\end{array}
\]

Notice that $h_{d+3}=\binom{d+2}{2}+3\le 2\binom{d-1}{2}$ for $d\ge 10$. Also, $h$ is differentiable up to degree $d+3$ and, \emph{relatively to $h_{d+3}$}, it grows maximally from degree $2d$ back to degree $d+3$. 

Our main task now is to show that $h$ is a level $h$-vector.  To this end, consider the infinite, eventually constant sequence $h' = (1,h_1,\dots,h_{d+2},h_{d+3}, h_{d+3}, \dots)$. The positive part of its first difference is $(1,2,3,\dots, d,d+1,1,1,1)$, where the last 1 occurs in degree $d+3$. This is easily seen to be the $h$-vector of a set of points $Z = Z_1 \cup Z_2$ in $\mathbb P^2(k)$, consisting of $Z_1 = \binom{d+1}{2}$ general points  in $\mathbb P^2(k)$ and $Z_2 = d+4$ general points on a line. Notice that $Z_1$ lies on no curve of degree $d-1$.

Let $Z_1 = \left\{ P_1,\dots,P_{\binom{d+1}{2}} \right\}$ and $Z_2 = \{ Q_1,\dots,Q_{d+4} \}$. For $1 \leq i \leq \binom{d+1}{2}$, let $L_i$ be the linear form dual to $P_i$ (i.e., if $P_i = [a,b,c]$ then $L_i = ax +by + cz$), and for $1 \leq i \leq d+4$, let $M_i$ be the linear form dual to $Q_i$. Thus, we can take the $L_i$ to be general linear forms in three variables, namely $x,y,z$, and without loss of generality, the $M_i$ to be general linear forms in two variables, say $y,z$.
We choose both $F_1$ and $F_2$ to be general linear combinations of suitable powers of the $L_i$ and $M_i$:
\[
\begin{array}{rcl}
F_1 & = & a_1 L_1^{2d} + \dots + a_{\binom{d+1}{2}} L_{\binom{d+1}{2}}^{2d} + b_1 M_1^{2d} + \dots + b_{d+4} M_{d+4}^{2d}; \\
F_2 & = & c_1 L_1^{2d} + \dots + c_{\binom{d+1}{2}} L_{\binom{d+1}{2}}^{2d} + d_1 M_1^{2d} + \dots + d_{d+4} M_{d+4}^{2d}.
\end{array}
\]

Consider the inverse system module generated by $F_1$ and $F_2$. We will assume for convenience that the base field $k$ has characteristic zero, in order to use differentiation, but an entirely similar argument works in arbitrary characteristic by employing the linear action described in Section 2 in the definition of inverse systems (see \cite{EI} for the standard but technical details). It suffices to show that the dimension of the linear span of the $j$-th partial derivatives of $F_1$ and $F_2$ is precisely $h_{2d-j}$, for all $1 \leq j \leq d-3$. Thus, since $I_Z$ is contained in the annihilator of $\langle F_1,F_2\rangle$, it is enough to prove that these derivatives are all linearly independent for as long as possible, subject only to this constraint. 

Each $j$-th partial derivative of either $F_1$ or $F_2$ is a suitable linear combination of the $L_i^{2d-j}$ and the $M_i^{2d-j}$, and we can view them as general linear combinations except that we must remember that any partial derivative involving $x$ will annihilate all powers of the $M_i$. We will consider the $j$-th derivatives, $1 \leq j \leq d-3$. We form a matrix whose rows are the coefficients of the $L_i^{2d-j}$ and the $M_i^{2d-j}$ arising in these derivatives, and the rank of this matrix will be the desired dimension.

\begin{center}

\begin{tikzpicture}[scale = .6]

\draw (0,0) -- (0,10);
\draw (0,0) -- (11,0);
\draw (0,10) -- (11,10);
\draw (11,0) -- (11,10);
\draw (5.5,10) -- (5.5,4.7);
\draw (0,4.7) -- (11,4.7);
\node at (1.5,10.8) {\scriptsize $L_1^{2d-j}$};
\node at (3,10.8) {\scriptsize $\dots$};
\node at (4.5,10.8) {\scriptsize $L_{\binom{d+1}{2}}^{2d-j}$};
\node at (6.5,10.8) {\scriptsize $M_1^{2d-j}$};
\node at (8,10.8) {\scriptsize $\dots$};
\node at (9.5,10.8) {\scriptsize $M_{d+4}^{2d-j}$};
\node at (-2.2,9) {\scriptsize $j$-th deriv. of $F_1$};
\node at (-2.2,8.4) {\scriptsize in $x,y,z$ with at};
\node at (-2.2,7.8) {\scriptsize least one $x$};
\node at (-2.2,6.8) {\scriptsize $j$-th deriv. of $F_2$};
\node at (-2.2,6.2) {\scriptsize in $x,y,z$ with at};
\node at (-2.2,5.6) {\scriptsize least one $x$};
\node at (8.2,7.6) {\Huge $0$};
\node at (2.7,7.6) {\Huge A};
\node at (5.4,2.4) {\Huge B};
\node at (-2.2,3.7) {\scriptsize $j$-th deriv. of $F_1$};
\node at (-2.2,3.1) {\scriptsize in $y,z$};
\node at (-2.2,2.1) {\scriptsize $j$-th deriv. of $F_1$};
\node at (-2.2,1.5) {\scriptsize in $y,z$};

\end{tikzpicture}

\end{center}

In the above matrix, the region marked $0$ consists entirely of zero entries. Since $d \geq 10$, the rank of the submatrix B is $2j+2$. The matrix $A$ has $2 \binom{j+1}{2}$ rows. One can check that for our range of $d$, the conditions
\[
\binom{d+1}{2} \leq 2 \binom{j+1}{2} \ \ \hbox{ and } \ \ 2j+2 \leq d+4
\]
are incompatible. Thus if $\binom{d+1}{2} \leq 2 \binom{j+1}{2}$, then the rank of the matrix is $\binom{d+1}{2} + (d+4)$. 
All together, we see that the  rank of our matrix is

$$\min \left  \{ \binom{d+1}{2} + (d+4), 2 \binom{j+1}{2} + (2j+2) \right \} =$$$$ \min \left \{ \binom{d+2}{2} +3, 2 \binom{j+2}{2} \right \} = h_{2d-j},$$
as desired. Hence $h$ is a level $h$-vector.

Therefore, Stanley's Lemma \ref{te} gives us that $(1,H_1,\dots,H_{2d},H_{2d+1}=1)$, where $H_i=h_i+h_{2d+1-i}$ for all $i\le 2d$, is a Gorenstein $h$-vector of codimension $H_1=3+2=5$. 

Notice that $H$ is unimodal, and that
$$H_{d-2}=\binom{d}{2}+\binom{d+2}{2}+3;$$
$$H_{d-1}=\binom{d+1}{2}+\binom{d+2}{2}+2;$$
$$H_d=2\binom{d+2}{2}+1.$$

Some simple arithmetic now shows that the first difference of $H$ equals $d-1$ in degree $d-1$ and $d$ in degree $d$, which violates Macaulay's Theorem \ref{macaulay}. This proves the case of any odd socle degree $e=2d+1\ge 21$.

The construction for even socle degrees $e\ge 20$ is similar, so we just sketch it here. We begin by fixing an integer $d\ge 10$, and consider the $h$-vector $h$ defined by 
\[
\begin{array}{lcl}
h_1 & = & 3, \\
h_2& = & 6, \\
& \vdots \\
h_d & = & \binom{d+2}{2}, \\
h_{d+1}& = & \binom{d+2}{2}+1, \\
h_{d+2}& = & \binom{d+2}{2}+2, \\
h_{d+3} & = & \min \left\{\binom{d+2}{2}+2,2\binom{d-2}{2}\right\}, \\
& \vdots \\
h_{2d-1-j} & = & \min \left\{\binom{d+2}{2}+2,2\binom{j+2}{2}\right\}, \\
& \vdots \\
h_{2d-2} & = & 6, \\
h_{2d-1} & = & 2.
\end{array}
\]

An argument similar to that in the previous case gives us that $h$ is level when $d\ge 10$. We can again construct its Gorenstein trivial extension  $H$ using Lemma \ref{te}. Its entries satisfy:
$$H_{d-2}=\binom{d}{2}+\binom{d+2}{2}+2;$$
$$H_{d-1}=\binom{d+1}{2}+\binom{d+2}{2}+1;$$
$$H_d=2\binom{d+2}{2}.$$

Thus, the first difference of $H$ again equals $d-1$ in degree $d-1$ and $d$ in degree $d$, against Theorem \ref{macaulay}. This completes the proof.
\end{proof}

\begin{remark}
We only remark here that, unfortunately,  the most natural guess for a general, purely numerical result that would include the level $h$-vectors $h$ of Theorem \ref{r} is not true. Namely, a socle degree $e$ $h$-vector, differentiable  up until a certain degree and then growing maximally from degree $e$ back to that degree, does not always need to be level, even in the non-Gorenstein case. For instance, $(1,3,4,5,6,\dots,2)$ (where the dots indicate any nonnegative number of 6's) is not level; see \cite{CI}, and also \cite{GHMS}, Example 7.3.
\end{remark}

\begin{example}\label{5}
The proof of Theorem \ref{r} shows that the following unimodal, non-SI $h$-vector of codimension 5 and socle degree 20 is Gorenstein:
$$(1,5,12,22,35,51,70,92,113,122,132,122,113,92,70,51,35,22,12,5,1).$$
It is obtained, using trivial extensions, from the socle degree 19 level $h$-vector
$$(1,3,6,10,15,21,28,36,45,55,66,67,68,56,42,30,20,12,6,2).$$
\end{example}

The argument of Theorem \ref{e} easily gives the following more refined result. We omit its proof since it is analogous to that of Corollary \ref{e4}.

\begin{corollary}\label{20}
For any integers $e\ge 20$ and $r\ge 5$, there exists a unimodal, non-SI Gorenstein $h$-vector of codimension $r$ and socle degree $e$.
\end{corollary}


Notice that while Theorem \ref{e} completely characterizes which socle degrees  admit a unimodal, non-SI Gorenstein $h$-vector, Theorem \ref{r}  closes all cases when it comes to the possible codimensions, \emph{except for $r=4$}. Here, finding new Gorenstein $h$-vectors seems extremely hard. So far, we do not even know any nonunimodal examples. For some recent results in codimension four, which heavily rely on the assumption that the characteristic of the base field be zero, see \cite{IS,MNZ,SS}.

We ask the following two-part question:

\begin{question}
\begin{itemize}
\item [(1)] For which integers $r$ and $e$ does there exist a unimodal, non-SI Gorenstein $h$-vector of codimension $r$ and socle degree $e$?
\item [(2)] Does the answer to  (1) depend on the \emph{characteristic of the base field}?
\end{itemize}
\end{question}

As far as (1) is concerned, the results of this paper imply a complete answer, among other cases, for $e\ge 20$ and any $r\neq 4$, and for $e=6$ and any $r\ge 10$. 

Case $r=4$ aside, it would be interesting to determine the least value of $e$ that gives a positive answer to part (1) when $r=5$, as well as the least $r$  when $e=6$. As for the latter problem, we note that while Theorem \ref{e} implies  $r\le 10$, standard but rather tedious methods show that $r$ is at least 6. However, determining whether codimension $r=9$ (and perhaps $r=8$ as well) allow the existence of unimodal, non-SI Gorenstein $h$-vectors of socle degree 6 might require a somehow different idea.

As for part (2), we believe the answer to be ``no,'' for every $r$ and $e$. (Note that all results in this paper are characteristic free.) In fact, we conjecture that the following important and much more general fact is true. We say that an $h$-vector $h$ is \emph{level in characteristic $p$} (here $p\ge 0$) if there exists a level algebra $k[x_1,\dots,x_r]/I$ having $h$-vector $h$ such that char$(k)=p.$

\begin{conjecture}\label{conj}
An $h$-vector is level in some characteristic if and only if it is level in \emph{every} characteristic.
\end{conjecture}

All results in the commutative algebra literature appear to go in the direction of Conjecture \ref{conj}. However, in general, a complete solution still seems out of reach, including in the Gorenstein case. It is difficult to even conjecture likely counterexamples using known techniques.


\section{Acknowledgements} This work was done while both authors were partially supported by  Simons Foundation grants  (\#309556 for Migliore, \#274577 for Zanello).



\begin{thebibliography}{llll}

\bibitem{A} C. A.  Athanasiadis: \emph{Ehrhart polynomials, simplicial polytopes, magic squares and a conjecture of Stanley}, J. Reine Angew. Math. \textbf{583} (2005), 163--174.
 
\bibitem{bass} H. Bass: \emph{On the ubiquity of Gorenstein rings}, Math. Z. \textbf{82} (1963), 8--28.

\bibitem{reniar} D. Bernstein and A. Iarrobino: \emph{A nonunimodal graded Gorenstein Artin algebra in codimension five}, Comm.  Algebra \textbf{20} (1992), no. 8, 2323--2336.

\bibitem{Bo} M. Boij: \emph{Graded Gorenstein Artin algebras whose Hilbert functions have a large number of valleys}, Comm. Algebra \textbf{23} (1995), no. 1, 97--103.

\bibitem{BL} M. Boij and D. Laksov: \emph{Nonunimodality of graded Gorenstein Artin algebras}, Proc. Amer. Math. Soc. \textbf{120} (1994), 1083--1092.

\bibitem{Bre} F. Brenti: \emph{Log-concave and unimodal sequences in algebra, combinatorics, and geometry: an update}, in ``Jerusalem Combinatorics '93,'' Contemporary Mathematics \textbf{178} (1994), 71--89.

\bibitem{BH} W. Bruns and J. Herzog: ``Cohen-Macaulay rings,'' Cambridge Studies in Advanced Math. \textbf{39}, Revised Ed. (1998), Cambridge, U.K..

\bibitem{CI} Y.H. Cho and A. Iarrobino: \emph{Inverse Systems of Zero-dimensional Schemes in ${\bf P}^n$}, J. Algebra \textbf{366} (2012), 42--77.

\bibitem{EI} J.\ Emsalem and A.\ Iarrobino, {\em Inverse system of a symbolic power, $I$}, J.\ Algebra {\bf 174} (1995), 1080--1090.

\bibitem{fl} R. Fr\"oberg and D. Laksov: \emph{Compressed Algebras}, Conference on Complete Intersections in Acireale, Lecture Notes in Mathematics \textbf{1092}, 121--151, Springer-Verlag (1984).

\bibitem{ger} A.V. Geramita: \emph{Inverse Systems of Fat Points: Waring's Problem, Secant Varieties and Veronese Varieties and Parametric Spaces of Gorenstein Ideals}, Queen's Papers in Pure and Applied Mathematics, No. 102, The Curves Seminar at Queen's, Vol. X, 3--114 (1996).

\bibitem{GHMS}  A.V. Geramita, T. Harima,  J. Migliore, and Y.S. Shin: ``The Hilbert function of a level algebra,'' Mem. Amer. Math. Soc. \textbf{186} (2007), no. 872.

\bibitem{ha} T. Harima:\emph{ Characterization of Hilbert functions of Gorenstein Artin algebras with the weak Stanley property}, Proc. Amer. Math. Soc. \textbf{123} (1995), no. 12, 3631--3638.

\bibitem{Hu} C. Huneke: \emph{Hyman Bass and Ubiquity: Gorenstein Rings}, Contemporary Math. \textbf{243} (1999), 55--78.

\bibitem{iarrobino} A. Iarrobino: \emph{Compressed Algebras: Artin algebras having given socle degrees and maximal length}, Trans. Amer. Math. Soc. {\bf 285} (1984), 337--378.

\bibitem{IKa} A. Iarrobino and V. Kanev: ``Power sums, Gorenstein algebras, and determinantal loci,'' Lecture Notes in Mathematics \textbf{1721}, Springer, Heidelberg (1999).

\bibitem{IS} A. Iarrobino and H. Srinivasan: \emph{Some Gorenstein Artin algebras of embedding dimension four, I: components of $PGOR(H)$ for $H=(1,4,7,\dots,1)$}, J. Pure Appl. Algebra  \textbf{201} (2005), 62--96.

\bibitem{MN} J. Migliore and U. Nagel: \emph{Reduced arithmetically Gorenstein schemes and simplicial polytopes with maximal Betti numbers}, Adv. Math. \textbf{180} (2003), 1--63.

\bibitem{MNZ} J. Migliore, U. Nagel, and F. Zanello:  \emph{A characterization of Gorenstein Hilbert functions in codimension four with small initial degree}, Math. Res. Lett. \textbf{15} (2008), no. 2, 331--349.

\bibitem{MNZ1} J. Migliore, U. Nagel, and F. Zanello:  \emph{On the degree two entry of a Gorenstein $h$-vector and a conjecture of
Stanley}, Proc. Amer. Math. Soc.  \textbf{136} (2008), no. 8, 2755--2762.

\bibitem{MNZ2} J. Migliore, U. Nagel, and F. Zanello: \emph{Bounds and asymptotic minimal growth for Gorenstein Hilbert functions}, J. Algebra \textbf{321} (2009), no. 5, 1510--1521.

\bibitem{MZ} J. Migliore and F. Zanello: \emph{Stanley's nonunimodal Gorenstein $h$-vector is optimal}, Proc. Amer. Math. Soc. \textbf{145} (2017), no. 1, 1--9.

\bibitem{P} R. Pandharipande: \emph{Three questions in Gromov--Witten theory}, in: Proceedings of the ICM, Vol. II,
Higher Ed. Press, Beijing (2002), 503--512.

\bibitem{Re} I. Reiten: \emph{The converse to a  theorem of Sharp on Gorenstein modules}, Proc. Amer. Math. Soc. \textbf{32} (1972), 417--420.

\bibitem{SS} S. Seo and H. Srinivasan: \emph{On unimodality of Hilbert functions of Gorenstein Artin algebras of embedding dimension four}, Comm. Algebra \textbf{40} (2012), no. 8, 2893--2905.

\bibitem{stanley} R.P. Stanley: \emph{Hilbert functions of graded algebras}, Adv. Math. \textbf{28} (1978), 57--83.

\bibitem{St-80} R.P. Stanley: \emph{The number of faces of a simplicial convex polytope}, Adv.  Math. \textbf{35} (1980), 236--238.

\bibitem{Stan} R.P. Stanley: \emph{Log-concave and unimodal sequences in algebra, combinatorics, and geometry}, Ann. New York Acad. Sci. \textbf{576} (1989), 500--535.

\bibitem{St0} R.P. Stanley:  \emph{A monotonicity property of $h$-vectors and $h^{*}$-vectors}, Europ. J. Combinatorics \textbf{14} (1993), 251--258.

\bibitem{St2} R.P. Stanley: ``Combinatorics and Commutative Algebra,'' Second Ed., Progress in Mathematics  \textbf{41} (1996), Birkh\"auser, Boston, MA.

\bibitem{Za0} F. Zanello: ``$h$-vectors and socle-vectors of graded artinian algebras,'' Ph.D. Thesis, Queen's University (2004).

\bibitem{Za3} F. Zanello: \emph{Stanley's theorem on codimension 3 Gorenstein $h$-vectors}, Proc. Amer. Math. Soc. \textbf{134} (2006), no. 1, 5--8.

\end{thebibliography}
\end{document}